\documentclass[reqno,12pt,letterpaper]{amsart}
\usepackage{amsmath,amssymb,amsthm,graphicx,mathrsfs,url}
\usepackage[usenames,dvipsnames]{color}
\usepackage[colorlinks=true,linkcolor=Red,citecolor=Green]{hyperref}

\def\?[#1]{\textbf{[#1]}\marginpar{\Large{\textbf{??}}}}

\def\smallsection#1{\smallskip\noindent\textbf{#1}.}

\newtheorem{prop}{Proposition}[section]

\newtheorem{lemm}[prop]{Lemma}

\numberwithin{equation}{section}

\DeclareMathOperator{\ad}{ad}
\DeclareMathOperator{\Res}{Res}
\DeclareMathOperator{\Spec}{Spec}
\DeclareMathOperator{\comp}{comp}

\DeclareMathOperator{\Imag}{Im}
\let\Im=\Imag

\DeclareMathOperator{\rank}{rank}

\let\Re=\Real

\DeclareMathOperator{\supp}{supp}
\DeclareMathOperator{\Vol}{Vol}
\DeclareMathOperator{\WF}{WF}

\def\WFh{\WF_h}

\def\squarebox#1{\hbox to #1{\hfill\vbox to #1{\vfill}}} 
\newcommand{\stopthm}{\hfill\hfill\vbox{\hrule\hbox{\vrule\squarebox
                 {.667em}\vrule}\hrule}\smallskip}

\title[Polynomial bounds on the number of Pollicott--Ruelle
resonances]{Sharp polynomial bounds 
on the number of Pollicott--Ruelle resonances}

\author{Kiril Datchev}
\email{datchev@math.mit.edu}
\address{Department of Mathematics, 77 Massachusetts Avenue, MIT,
Cambridge, MA 02139}
\author{Semyon Dyatlov}
\email{dyatlov@math.berkeley.edu}
\author{Maciej Zworski}
\email{zworski@math.berkeley.edu}
\address{Department of Mathematics, University of California,
Berkeley, CA 94720, USA}

\begin{document}

\begin{abstract}
We give a sharp polynomial bound on the number of Pollicott--Ruelle
resonances. These resonances, which are complex numbers 
in the lower half-plane, appear
in expansions of correlations for Anosov contact flows. 
The bounds follow the tradition of upper bounds on 
the number of scattering resonances and improve a
recent bound of Faure--Sj\"ostrand. The complex scaling
method used in scattering theory is replaced by 
an approach using exponentially weighted spaces
introduced by Helffer--Sj\"ostrand in scattering theory and by
Faure--Sj\"ostrand
in the theory of Anosov flows.
\end{abstract}

\maketitle


\section{Introduction and statement of the results}
\label{in}

Pollicott--Ruelle resonances appear in correlation expansions for certain
chaotic dynamical systems \cite{Po,Ru}.
Recently Faure, Roy, and Sj\"ostrand \cite{f-r-sj,f-sj}
explained how some aspects of Anosov dynamics can be 
analyzed using microlocal methods of scattering theory. 
As an application of that point of view
Faure and Sj\"ostrand \cite{f-sj} proved the
following polynomial upper bound for the number of Pollicott--Ruelle resonances, denoted
$ \Res ( -iV  ) $,  
of a contact Anosov flow, $ \varphi_t = \exp t V $, 
on a $ n$-dimensional compact smooth manifold:
\begin{equation}
\label{eq:b1}
\# \{ \lambda \in \Res ( -iV  ) \; : \;  | \Re \lambda - E | \leq \sqrt E ,
\ \ \Im
\lambda > - \beta \} = o ( E^{ n - \frac12} ) ,
\end{equation}
for any $ \beta $.
See~\cite[Theorem~1.8]{f-sj} for a detailed statement.

In this paper we develop their approach further using 
recent advances in resonance counting \cite{da-dy,sj-n-z,sj-z}. 
This gives the following improvement of \eqref{eq:b1}
\begin{equation}
\label{eq:b2}
\# \{ \lambda \in \Res ( -i V  ) \; : \; | \Re \lambda  - E | \leq \sqrt E  ,
\ \  \Im
\lambda > - \beta \} ={ \mathcal  O } ( E^{\frac{ n} 2} ).
\end{equation}
This estimate is a consequence of an optimal bound
holding in smaller energy intervals \eqref{eq:b31}.

We briefly review the setting referring to \cite[\S 1.1]{f-sj} and 
\cite{b-t}  for
more details and numerous references to earlier works, in particular
in the dynamical systems literature.

Let $ X  $ be a compact smooth manifold of odd dimension 
$ n \geq 3 $, and let $ \varphi_t : X \to X $ be an 
{\em Anosov flow} on $ X $.  We assume that there
exists a {\em contact form} $ \alpha \in
C^\infty ( X , T^*X ) $ compatible with that flow.
This means that for $ E_u ( x ), E_s ( x ) \subset T_x X  $, stable and unstable
subspaces at $ x $,  we have
\begin{equation}
\label{eq:alsu}   \ker ( \alpha ( x ) ) = E_u ( x) \oplus E_s ( x ) \,, \ \  
d \alpha ( x ) |_{ E_s ( x) \oplus E_u ( x ) } \ \text{ is
  nondegenerate.} 
\end{equation}
If $ V \in C^\infty ( X , TX ) $ is the generator of the flow then
$ \alpha ( V ) \neq 0 $ and we can modify $ \alpha $ so that 
$ \alpha ( V ) = 1 $, the assumption we make. In particular,
$ {\mathcal L}_V \alpha = 0 $. 

The volume form on $ X $ is given by 
\[ dx := \alpha \wedge ( d \alpha )^{ \frac{ n-1} 2 } ,\ \ {\mathcal
  L}_V dx = 0 , \]
and
\[ P := {\textstyle {\frac h i}} V , \ \ P : L^2 ( X, dx ) \to L^2 ( X , dx ) ,\]
is a symmetric first order semiclassical differential operator -- see
\cite[\S 14.2]{e-z}. As shown
in \cite[Appendix~A.1]{f-sj} it is essentially self-adjoint. The
addition
of the semiclassical parameter, although trivial, makes the final
argument more natural.

The Pollicott--Ruelle resonances of $ P $, or $ \varphi_t $, are defined
as eigenvalues of $ P$ acting on {\em exponentially weighted
spaces}, $ H_{tG } $, introduced in scattering theory by 
Helffer--Sj\"ostrand \cite{h-sj} and in the context of this paper by 
Faure--Sj\"ostrand \cite{f-sj} (see also Faure--Roy--Sj\"ostrand
\cite{f-r-sj} for an earlier version for Anosov diffeomorphisms). 
The construction of these spaces, denoted by $H^m$ in~\cite{f-sj},
will be reviewed in \S \ref{op} below.
The main point is the following fact given in \cite[Theorem 1.4]{f-sj}:
\begin{gather*}
 P - z : {\mathcal D}_{tG} \to H_{tG} \ \text{ is a
   Fredholm operator for $ \Im z > -th/C , \ t \gg 1 $, } \\
{\mathcal D}_{tG} := \{ u \in H_{t G} : P u \in H_{tG} \}, \text{ with
$ Pu $ defined in the sense of distributions.}
\end{gather*}
By Analytic Fredholm theory (see for instance~\cite[Theorem~D.4]{e-z})
the resolvent $(P-z)^{-1}:H_{tG}\to {\mathcal D}_{tG}$ is meromorphic with
poles of finite rank, which are called Pollicott--Ruelle resonances.
These resonances are independent of
$ t $ and depend only on quantitative properties of the weight $ G$,
see~\cite[Theorem~1.5]{f-sj}. See \S \ref{oop} for some heuristic
ideas behind this construction.

The bound \eqref{eq:b2} is a consequence of a bound in smaller energy 
intervals given in our main result:

\medskip
\noindent
{\bf Theorem.} {\em Let $ X $ be a compact smooth manifold with an
  Anosov flow $ \varphi_t : X \to X $. Let $ P $ be the first order
self-adjoint  operator such that $ i P / h $ is the generator of $
\varphi_t $, and let $ \Res ( P ) $ be the set of resonances
of $ P $. Then for any $ C_0 > 0 $, 
\begin{equation}
\label{eq:b3}
\# \Res ( P ) \cap D ( 1 , C_0 h ) = {\mathcal O } ( h^{- \frac{n-1} 2} ) , 
\end{equation}
where $ D ( z, r ) = \{ \zeta : |\zeta - z | < r \}$.}
\medskip

\noindent
{\bf Remarks} (i) The bound \eqref{eq:b2} was predicted 
in remarks after \cite[Theorem 1.8]{f-sj} 
and is an immediate consequence
of \eqref{eq:b3}. Rescaling $ \lambda = z/h $ we rewrite
\eqref{eq:b3} as 
\begin{equation}
\label{eq:b31}
\# \{ \lambda \in \Res ( -i V  ) \; : \; | \Re \lambda  - E | \leq
1   , \  \Im
\lambda > - \beta \} ={ \mathcal  O } ( E^{\frac{ n-1} 2} ) . 
\end{equation}

\medskip
\noindent
(ii) When $ X = S^*M $ where $ M $ is a compact surface of constant
negative curvature, Pollicott--Ruelle resonances coincide with the zeros
of the Smale zeta function -- see for instance \cite[\S 5.2, Figure
1]{Leb}. Except of a finite number these are then given by 
\[  \lambda =  z  - i ( k + \textstyle{\frac 12 } ) 
 , \  \ \ k \in \mathbb N , \ \ \  z^2  \in \Spec ( - \Delta_M -
{\textstyle \frac 14 }) , \]
where $ \Delta_M  $ is the Laplace-Beltrami operator on $ M$.
In that case the spectral asymptotics \cite{Ber},\cite{Ran} give, for $ \beta > 0
$, 
\[  \# \{ \lambda \in \Res ( P ) \; : \;  | \Re \lambda - E | \leq 1 , \ \Im
\lambda > - \beta \} = [ \beta + {\textstyle{\frac12}} ]  \frac{ \Vol (
  M ) } { \pi} E  + {\mathcal O} \left(\frac E { \log E } \right). \]
In this case $ n = \dim X = 3 $ which shows the optimality of 
\eqref{eq:b3}.

\medskip
\noindent
(iii) In a recent preprint \cite{fat}, Faure and Tsujii consider 
the case of partially hyperbolic diffeomorphism conserving a smooth contact
form and obtain a description of the spectrum in terms of ``bands",
corresponding to fixed values of $ k$ in the example above. 
The asymptotics given in \cite[Theorem 1.19]{fat} are in agreement with 
the upper bound \eqref{eq:b3}.

\section{Outline of the proof}
\label{oop}

The proof of the bound \eqref{eq:b3} is based on combining the 
arguments of \cite{f-sj} with the arguments of
\cite{sj-z},\cite{da-dy}. The paper relies heavily on 
technical results from these earlier works and we provide
specific references in the text. Here we will motivate the 
problem and outline
the general idea of the proof of \eqref{eq:b3}.

The basic analogy between analysis of flows and semiclassical
scattering theory lies in the fact that for a flow $ \varphi_t = 
\exp t V : X \to X $, 
\begin{equation}
\label{eq:qp}
\varphi_t^* u = e^{ i t P / h } u , \ \ u \in C^\infty ( X ) , \ \ \ 
\varphi_t \circ \pi  = \pi \circ \exp t H_p , 
\end{equation}
where $ p ( x, \xi ) = \xi ( V_x )  $ is the symbol of the differential operator $ P $, 
$ \pi : T^* X \to X $ is the canonical projection, $ H_p $ is the 
Hamilton vector field of $ p $, and $\varphi_t^*$ is the pullback operator:
$\varphi_t^*u:=u\circ\varphi_t$.

The key object in scattering theory 
is the trapped set at energy $ E $:
\begin{equation}  K_E = \{  ( x, \xi ) \in p^{-1} ( E ) : \exp ( t H_p ) ( x , \xi )
\not \to \infty , \ \ t \to \pm \infty \}. 
\end{equation}
Here we note that $ \infty $ means fiber infinity $ \xi \to \infty $
in $ T^* X $.  It is the only ``infinity'' in our setting as $ X $ is compact.

Just as in scattering theory the spectrum of the unitary operator $ \exp ( i t P/h ) $ is
the unit circle $ \mathbb S^1 $, so to expand the correlations 
\[   \langle \varphi^*_t f , g \rangle = \langle e^{ i t P / h } f ,
g \rangle, \ \ f , g \in C^\infty ( X ) , \]
into modes of decay,  $ P $ has to be considered on a modified space, 
containing  $ C^\infty ( X ) $,  and such that $ P - z $ becomes a
Fredholm operator for $ \Im z > - A  h $. Roughly speaking, this may
provide an expansion into modes with errors of size $ e^{ - A t }$ as
$ t \to \infty $ -- see \cite[Theorem 2]{f-r-sj} for the
case of Anosov diffeomorphism (where $ t = n $ is discrete and $ \epsilon = e^{-A} $). 

Following earlier works of Aguilar-Combes, Balslev-Combes, and Simon, 
Helffer-Sj\"ostrand \cite{h-sj} introduced 
an approach based on an {\em escape function}, that is a function on 
$ T^* X $, such that $ H_p G \leq 0 $ everywhere and $ H_p G < 0 $ near the
infinity of the characteristic set of $ p $ (in fact in as large a set
as possible). It turns out that the inhomogeneous Sobolev spaces
used in the works of 
Baladi \cite{bal}, Tsujii \cite{ts1,ts2,ts3},  Blank, Butterley, Gou\"ezel, Keller, and
Liverani \cite{b-k-l,b-l,g-l,liv1,liv2}
can be reinterpreted this way. As explained in \S
\ref{op} below,
\begin{gather*}  P - z : {\mathcal D}_{tG} \to H_{tG} , 
\text{ is a
   Fredholm operator for $ \Im z > -th/C , \ t \gg 1 $, } \\
  H_{tG} := \exp ( t G^w
( x, h D ) ) L^2 ( X ) , 
 \end{gather*}
where $ G $ is an escape function and $ {\mathcal D}_{t G } $ is the
domain of $ P $ in $ H_{tG} $.   Working on $ H_{tG} $ is equivalent
to working on $L^2(X)$ with $ P $ conjugated by the exponential weight and
the basic idea comes from (see \eqref{eq:PtG} for a precise statement)
\begin{equation}
\label{eq:pG}
e^ {t G^w ( x, h  D) } P e^{ - t G^w ( x, h D )  }  \sim 
P + i t h ( H_p G )^w ( x, h D ) .
\end{equation}
As shown in \cite[\S 3]{f-sj} negativity of $ H_p G $ near infinity 
implies the Fredholm property of $ P - z $ for $ \Im z > - t h/C $.
The eigenvalues of  $ P$, which are now
complex and lie in $ \Im z \leq 0 $, are called {\em Pollicott--Ruelle
  resonances}, 
and we denote them by $ \Res ( P )$. 

In scattering theory polynomial bounds on the number of 
resonances were first obtained by Melrose. Sharp 
bounds in odd dimensions were given by  Melrose \cite{Mel3} 
for obstacles and by Zworski \cite{Zw} for compactly supported
potentials; the even dimensional sharp bounds were later obtained
by Vodev \cite{Vo}.

The seminal work of Sj\"ostrand \cite{sj-d} and numerous mathematics
and physics papers that followed (see \cite{da-dy} and \cite{sj-n-z}
for references) then indicated that the exponent in the upper bound
on the number of resonances near energy, say energy  $ E =1
$\footnote{Since $ P = - i h V $ the problem is clearly homogeneous
  and we can work near any non-zero energy level. The high energy
limit corresponds to $ h \to 0 $.}   should be related to the 
dimension of the trapped set $ K_E $. In our case that dimension is
in fact integral (see \eqref{e:ke0}):
\[  \dim K_E = n = 2 \mu + 1 ,  \ \  \mu = \frac { n -1 } 2 \in
\mathbb N, \]
For bounds in regions of size $ h $, the result of \cite{sj-z}
(following an earlier small neighbourhood bound by 
Guillop\'e-Lin-Zworski for Selberg zeta functions of Schottky groups
-- see \cite{da-dy})
suggests
that 
\begin{equation}
\label{eq:bm}   \# \Res ( P ) \cap D ( 1 , C_0 h ) = {\mathcal O } ( h^{- \mu }  )
. \end{equation}
This is precisely the bound given in \eqref{eq:b3} which by rescaling
translates to the bound \eqref{eq:b31}.  In many situations the
interest in \eqref{eq:bm} lies in the fact that $ \mu $ may not be 
an integer.

To obtain \eqref{eq:bm} we need to modify $ G $ so that, 
when using \eqref{eq:pG}, we can invert the conjugated $ P - 1 $ 
microlocally on a larger set. 
More precisely we introduce the additional conjugations in
  \eqref{e:p-t} below;
that is done as in \cite{sj-z} with the
modifications
presented in \S \ref{cef}. The ideas behind that strategy are 
explained in \cite[\S 2]{sj-z}. To localize in a neighbourhood
of size $ h$ a second microlocal argument is needed and 
we followed the functional calculus approach presented in 
\cite{da-dy}. 

Eventually, these constructions produce a modified
operator $   \widetilde P_t - z $ given in Main Lemma \ref{l:main}
which is invertible for $ z \in D ( 1 , C_0 h ) $,  and which differs
from the conjugated operator by an operator $ - i t h A  $,  microlocally 
localized in an $ {\mathcal O} ( \sqrt h ) $  neighbourhood of $ K_1 
$ with an additional $ {\mathcal O } (h) $ localization in the
direction of  $ dp $.   

The basic semiclassical intuition then dictates that
the number of resonances of $ P $ in $ D ( 1, C_0 h )$  (which 
are the same as the eigenvalues of the conjugated operator)
is given by the phase space volume occupied by $ A $ multiplied by 
$ h^{-n} $. That volume is estimated by  $ h $ (due to the energy localization)
times the volume of an  $ {\mathcal O} ( \sqrt h ) $ neighbourhood of the
smooth set $ K_1 $ inside $ p^{-1} ( 1 ) $. Since $ K_1 $ has
dimension  $ 2 \mu + 1 $, its codimension inside $ p^{-1} ( 1 )$  is given
by $ 2 ( n - \mu - 1 ) $. This gives the following bound:
\[  h^{-n} \times h \times h^{ \frac 12 ( 2 ( n - \mu - 1 ) ) } =
h^{-\mu } , \]
proving \eqref{eq:bm}.

\section{Microlocally weighted spaces and discrete spectrum of the
  generator of the flow}
\label{op}

In this section we review, in a slightly modified form, the
construction of Faure-Sj\"ostrand \cite{f-sj} which provides 
Hilbert spaces on which $ P - z  $ is a Fredholm operator
for $ \Im z  > - \beta h $. 

Following \cite{h-sj},\cite{f-r-sj} the crucial component
is the construction of an escape function on $ T^* X $, that
is a function $ G $ for which $ H_p G \leq 0 $ everywhere, with strict 
inequality on a large set.  

The decomposition into neutral (one dimensional), 
stable and unstable subspaces is given by (here $E_0(x)$
is spanned by $V$)
\[  T_x X = E_0 (x) \oplus E_s( x ) \oplus E_u ( x ) .\]
The dual decomposition is
obtained by taking $ E_0^* ( x ) $ to be the annihilator of
$ E_s ( x) \oplus E_u ( x ) $, $ E_u^* ( x) $ the annihillator
of $ E_u ( x ) \oplus E_0 ( x ) $, and similarly for $ E_s^*( x) $.
That makes $E_s^*( x ) $ dual to $E_u ( x ) $, $ E^*_u (x) $ dual to 
$ E_s ( x ) $, and $ E_0^* ( x ) $ dual to $ E_0 ( x ) $. The fiber
of the cotangent bundle decomposes as 
\begin{equation}\label{e:cotdec}
T_x^*X = E_0^*(x)\oplus E_s^*(x)\oplus E_u^*(x) .
\end{equation}
We recall that the distributions $E_s^* (x) $ and $E_u^* ( x ) $ have only H\"older regularity, but
$E_0^* ( x ) $ and $E_s^*(x)\oplus E_u^* ( x ) $ are smooth, and that 
$E_0^*=\mathbb R \alpha$ -- see \eqref{eq:alsu}.

Let $|\cdot|$ be any smooth norm on the fibers of $T^*X$ such that
the norm of $\alpha$ and
the dual norm of $V$ are equal to $1$, so that in particular
$$
\{|\xi|\leq 1/2\}\cap p^{-1}(1)=\emptyset.
$$
Here $p=p(x,\xi)$ is the classical Hamiltonian corresponding to $P$, i.e. the linear function
on the fibers of $T^*X$ defined by $V$.

The classical flow on $ T^*X $ is explicit in terms of $ \varphi_t $:
\begin{equation}
\label{eq:Hpex}  \exp t H_p ( x , \xi ) = ( \varphi_t ( x ) , (D_x \varphi_t ( x)^T )^{-1}\xi) \, 
\end{equation}
It follows from the hyperbolicity of the flow (see~\cite[(1.13)]{f-sj})
that for some constants $C>0$ and $\theta>0$ and for all $t>0$,
\begin{equation}
\label{e:hyperbolicity}
\begin{gathered}
|\exp t H_p (\rho)|\leq Ce^{-\theta t}|\rho|,\quad \rho\in E_s^*,\\
|\exp -t H_p (\rho)|\leq Ce^{-\theta t}|\rho|,\quad \rho\in E_u^*.
\end{gathered}
\end{equation}
The {\em trapped set}, that is the set of $ ( x,\xi )$ which stay 
in a compact subset (depending on $ ( x, \xi ) $) for all $  t \in
\mathbb R$
is  given by 
\begin{equation}\label{e:ke0} 
K = E_0^* = \bigcup_{x\in X } E_0^* ( x ) \subset T^* X
\end{equation}
and is a smooth submanifold of $ T^*X $, which is symplectic
away from the zero section.
Indeed, since the decomposition \eqref{e:cotdec} is invariant under $\varphi_t$,
we may apply \eqref{e:hyperbolicity} with $\exp(\mp t H_p(\rho))$ in place of $\rho$ to show
$K \subset E_0^*$, and $E_0^* \subset K$ follows from $E_0^* = \mathbb R \alpha$.
The energy slice of the trapped set is defined as
\begin{equation}
  \label{e:k-e}
K_1 = p^{-1} ( 1 ) \cap E^*_0 .
\end{equation}

We denote by $ S ^k ( T^*X ) $ the standard space of symbols
used in \cite[\S 14.2]{e-z} and by $ S^{ k+ } ( T^*X ) $
the intersection of  $ S^{k+\epsilon } $ for all $ \epsilon > 0 $.
The class of semiclassical pseudodifferential operators
corresponding to $S^k(T^*X)$ is denoted by  $\Psi^k(T^*X)$ -- 
 see~\cite[\S 14.2]{e-z} and \cite[\S 3.1]{da-dy} for a review of the semiclassical
notation used below. We also write $A\in\Psi^{k+}$ to denote that $A\in\Psi^{k+\varepsilon}$
for every $\varepsilon>0$.

Following~\cite{f-sj}, we construct the weight function $G$:
%
%
\begin{lemm}
  \label{l:escape-f-sj}
Take any conic neighborhoods $U_0,U_0'$ of $E_0^*$, with $U_0\Subset U'_0$
and $U'_0\cap (E_u^*\cup E_s^*)=\emptyset$.
Then there exist real-valued functions $m\in S^0(T^*X),f_0\in S^1(T^*X)$
such that
\begin{enumerate}
\item $m$ is
positively homogeneous of degree 0 for $|\xi|\geq 1/2$,
equal to $-1,0,1$ near the intersection of $\{|\xi|\geq 1/2\}$ with
$E_u^*,E_0^*,E_s^*$, respectively, and
\begin{equation}
  \label{e:hpm}
H_p m<0\text{ near }(U'_0\setminus U_0)\cap \{|\xi|>1/2\},\quad
H_p m\leq 0\text{ on }\{|\xi|>1/2\};
\end{equation}
\item $\langle\xi\rangle^{-1}f_0\geq c>0$ for some constant $c$;
\item the function $G=m\log f_0$
satisfies for some constant $c$,
\begin{equation}
  \label{e:hpG}
H_pG\leq -c<0\text{ on }\{|\xi|\geq 1/2\}\setminus U_0,\quad
H_pG\leq 0\text{ on }\{|\xi|\geq 1/2\}.
\end{equation}
\end{enumerate}
\end{lemm}
\begin{proof}
The existence of $m$ follows from~\cite[Lemma~1.2]{f-sj},
where we rescale the parameter $\xi$ to map the region $\{|\xi|\leq R\}$
of~\cite{f-sj} into $\{|\xi|\leq 1/2\}$, and use the function
$\sqrt{1+f^2}$ of~\cite{f-sj} as $f_0$. The inequality~\eqref{e:hpm}
follows directly from the proof of~\cite[Lemma~1.2]{f-sj}, if we
choose the neighborhoods $\widetilde N_u,\widetilde N_0,\widetilde N_s$
there so that $U'_0\cap(\widetilde N_u\cup\widetilde N_s)=\emptyset$
and $\widetilde N_0\subset U_0$.
\end{proof}
%
%
Note that~\cite{f-sj} needed $H_pG<-C<0$ on $\{|\xi|\ge 1/2\} \setminus U_0$
for a large constant $C$; in this paper,
we instead multiply $G$ by a large $t>0$ in the conjugation. The neighborhoods
$U_0,U'_0$ will be chosen in \S\ref{s:global-escape}.

The function $G$ satisfies derivative bounds
\begin{equation}
\label{eq:strG}
G  = {\mathcal O} ( \log \langle \xi \rangle ) , \ \ 
 \partial^\alpha_x\partial^\beta_\xi H^k_p G
= {\mathcal O} \left( \langle\xi\rangle^{-|\beta|+} \right) 
, \ \ | \alpha | + | \beta| + k \geq 1 \,.
\end{equation}
In particular, $ \partial^\alpha_x\partial^\beta_\xi H^k_p G \in 
S^{- | \beta| + } $.

We now use \cite[\S 8.3]{e-z} (the small modification to take into
account the symbol classes $ S^m $ and $ S^{m + } $ is done as in 
\cite[\S 9.3, \S 14.2]{e-z}) and define
\begin{equation} 
\label{eq:defHG}
H_{tG} ( X ) := \exp ( - t G^w ( x, h D ) ) L^2 ( X , dx ) . 
\end{equation}
Note that this space is topologically isomorphic, with the norm of the isomorphism
depending on $h$, to the nonsemiclassical space used in~\cite{f-sj}:
$$
H^m(X) := \exp( -t G^w(x, D))L^2(X, dx).
$$
Indeed, for $|\xi|>1/2$ the difference 
$$
G(x,\xi)-G(x,h\xi)=m(x,\xi)\log(f_0(x,\xi)/f_0(x,h\xi))
$$
and it its derivatives, are bounded uniformly in $(x,\xi)$ for any
fixed $h$, so that equality of the spaces follows from 
\cite[Theorem~8.8]{e-z} applied with $ h $ fixed -- 
see also~\cite[\S 5.2]{f-sj}.

The domain of $ P $ acting on $ H_{tG} $ is defined as
\begin{equation}
\label{eq:defDG}
{\mathcal D}_{tG} := \{ u \in {\mathcal D}' ( X )  \; : \; u, P u \in H_{tG } \}.
\end{equation}
The action of $ P $ on $ H_{t G}$ is equivalent to the action of the
more natural operator $P_{tG} $ on $ L^2 $:
\begin{equation}
\label{eq:PtG}
\begin{split}  P_{tG} & := e^{ t G^w } P e^{ - t G^w } 
= \exp (t  \ad_{ G^w } ) P
\\ & = \sum_{k=0}^N \frac{ t^k } { k!} \ad_{G^w }^k P + R_N ( x, h D)
, \ \ \ 
R_N \in h^{N+1} S^{ - N + } . 
\end{split}
\end{equation}
One way to see the validity of \eqref{eq:PtG} is to note
that the operators $ e^{ \pm t G^w } $ are pseudodifferential
operators \cite[Theorem 8.6]{e-z}
 and hence the pseudodifferential
calculus applies directly \cite[Theorem 9.5, Theorem 14.1]{e-z}. To 
show that $R_N \in h^{N+1} S^{ - N + }$, write $R_N$ as a sum of $2t+1$ many 
terms of a Taylor series plus an integral remainder which can be
analyzed as in, for example, 
\cite[Lemma 7.2]{da-dy}.

Using this expansion we can follow arguments in \cite[\S 3]{f-sj} 
to show 
\begin{prop}
\label{p:1}
For $ P_{tG} $ defined by \eqref{eq:PtG} we have:

\noindent
i) $  P_{tG} - z : \mathcal D(P_{tG})\to L^2 $ is a Fredholm
operator of index zero for $ \Im z > - th / C $. Here $\mathcal D(P_{tG})$
is the domain of $P_{tG}$.

\noindent
ii) $P_{tG}-z$ is invertible for $\Im z>Ch$ and $C$ large enough.

\end{prop}

\section{Construction of the escape function}
\label{cef}

In this section we modify the escape function of Faure--Sj\"ostrand
near the trapped set.  It will be
quantized to become the  operator $F$ appearing in Main Lemma \ref{l:main}
below. We will use symbols depending on two semiclassical parameters
$h,\tilde h$, see \S~\ref{s:proof} for details.

\subsection{Construction near the trapped set}
\label{s:escape}

We start with an escape function $\hat f$ defined in a neighborhood of $K_1$
and with $H_p\hat f\leq -c<0$ away from a $C(h/\tilde h)^{1/2}$ sized 
neighborhood of the trapped set $K$. This is
a modification of the construction in~\cite[\S 7]{sj-z}, based on 
an earlier construction in~\cite[\S 5]{sj-d}.  The changes come 
from a different structure of the incoming and outgoing manifold which 
we now define:
\begin{equation}
\label{eq:defG}
 \Gamma_\pm = \{ ( x, \xi ) \;  : \;
\exp ( t H_p ) ( x , \xi ) \not \to \infty , \ \ t \to \mp \infty \}. 
\end{equation}
We note that $ \infty $ refers to the fiber infinity of $ T^* X $. We
see that 
\[  K  =   \Gamma_+
\cap   \Gamma_- , \] 
and that by~\eqref{e:hyperbolicity},
\begin{gather}
\label{eq:Gam}
\begin{gathered} 
  \Gamma_+ = \bigcup_{x \in X }   \Gamma_{+, x},  \ \ \ 
  \Gamma_- = \bigcup_{x \in X }   \Gamma_{-,x} , \\
  \Gamma_{+,x} :=  E_0^* ( x ) \oplus E_u^*  ( x
  ), \ \ \ 
  \Gamma_{-,x} :=  E_0^* ( x ) \oplus E_s^*  ( x
  ) . 
\end{gathered}
\end{gather}
 This provides a continuous but typically non-smooth foliation of $   \Gamma_\pm $ by 
smooth (linear) manifolds. We note that for $ ( x ,\xi ) \in   K $, the
(linear) leaves of the two folliations intersect cleanly with a fixed
excess equal to $ n $, the dimension of $ X $, 
\[  ( E_0^* ( x ) \oplus E_u^*  ( x ) ) \cap  ( E_0^* ( x ) \oplus E_s^*  ( x )
)  = E_0^* ( x ) , \]
see~\cite[Appendix~C.3]{hor}.
This, rather than the transversality of leaves, assumed in~\cite[\S 5]{sj-d} and~\cite[\S 7]{sj-z}
constitutes the only difference in the construction. Nevertheless the
basic facts established in~\cite[\S 5]{sj-d} still hold. To state
them we use
the notation $ f \sim g $ if, for a constant $ C $, $ f/C  \leq g \leq
C f $.
\begin{lemm}
\label{l:f-0}
Let $ d  $ be a distance function on a neighbourhood of $ K_1 \subset T^*
X $.
For $ \rho $ in a neighbourhood of $ K_1 $, we have
\begin{equation}
\label{eq:f1}
d ( \rho,   K ) \sim d ( \rho,   \Gamma_+ ) + d ( \rho,
  \Gamma_- )  . 
\end{equation}
Also, there exists a constant $ C $ such that for any $ \tau \geq 0$ 
we can find an open neighbourhood $ \Omega_\tau $ of $ K_1 $
such that 
\begin{equation}
\label{eq:f2}  d ( \exp (\pm \tau H_p) ( \rho ) , \Gamma_\pm ) \leq C
e^{ - \tau / C } d( \rho,   \Gamma_\pm ) ,\
\rho\in\Omega_\tau. \end{equation}
\end{lemm}
\begin{proof}
The cleanness with a fixed excess (affine spaces always intersect
cleanly) 
shows that for $ x \in X $ and $ \rho $ close to $ K_1 $ we
still have a uniform statement,
\[  d ( \rho,   \Gamma_{+ , x } \cap   \Gamma_{-, x} ) \sim 
d ( \rho,   \Gamma_{- , x } ) + d ( \rho,   \Gamma_{ + , x })  . \]
Hence \eqref{eq:f1} follows as in the proof of \cite[Lemma 5.1]{sj-d}.

To obtain \eqref{eq:f2} 
we choose a euclidean distance $ d_y $  on fibers $ T_y ^* X $,
depending smoothly on $ y $. From the continuity of $ x \mapsto \Gamma_{\pm , x
  } $ we see that for $ ( x, \xi ) $ in a bounded set, 
\begin{equation}
\label{eq:bounded} d  ( ( x , \xi ) , \Gamma_\pm ) \sim d_x ( \xi,   \Gamma_{\pm, x } ).
\end{equation}
We now fix a bounded neighbourhood of $K_1$,  $ \Omega $, in which 
\eqref{eq:bounded} is valid uniformly and define 
\[  \Omega_\tau := \bigcup_{ | t | \leq \tau } 
\exp t H_p ( \Omega ) . \]
Then for $ ( x, \xi ) \in \Omega_\tau $, \eqref{eq:Hpex} shows 
that
\[ \begin{split} 
d ( \exp ( \pm \tau H_p ) ( x, \xi ) ,   \Gamma_\pm )  & \sim
 d_{\varphi_{\pm \tau } ( x ) }   ( ( D_x \varphi_ {\pm \tau} (x)
^T)^{-1}  \xi ,   \Gamma_{\pm , \varphi_{\pm \tau} ( x )}  ) , 
\end{split} \]
with constants independent of $ \tau $.

Hence with  $ C $ independent of $ \tau $, 
and $ ( x, \xi ) \in \Omega_\tau $, we have 
\begin{equation} 
\label{eq:rede}  d_{\varphi_{\tau} ( x ) } (( D_x \varphi_ {\tau} (x)
^T)^{-1}  \xi ,   \Gamma_{+ , \varphi_{ \tau} ( x )}  ) 
\leq C e^{-\tau/C } d_x ( \xi , \Gamma_{+ , x } ) , \ \ ( x, \xi )
\in \Omega_\tau . 
\end{equation}
(We state this for $ + $, the other case being analogous.) 
We write the unique decomposition $ \xi = \xi_u + \xi_s + \xi_0 $, 
$ \xi_\bullet \in E_\bullet^* ( x ) $, so that by the invariance of the
subspaces $  E_\bullet^*  $,
\begin{gather*}
 ( D_x \varphi_ {\tau} (x) ^T)^{-1}  \xi = 
( D_x \varphi_ {\tau} (x) 
^T)^{-1}  \xi_u   + ( D_x \varphi_ {\tau} (x)
^T)^{-1}  \xi_s   + ( D_x \varphi_ {\tau} (x)
^T)^{-1}  \xi_0 , \\ ( D_x \varphi_ {\tau} (x)
^T)^{-1}  \xi_\bullet \in E^*_\bullet ( \varphi_\tau ( x ) ) . 
\end{gather*} 
This means that 
\[  d_{\varphi_{\tau} ( x ) } (( D_x \varphi_ {\tau} (x)
^T)^{-1}  \xi ,   \Gamma_{+ , \varphi_{\tau} ( x )}  )  \sim 
\|  ( D_x \varphi_ {\tau} (x)
^T)^{-1}  \xi_s \| , \ \  
d_x ( \xi , \Gamma_{+ , x } )  \sim 
\| \xi_s \|. \]
The estimate \eqref{eq:rede} then follows
from the Anosov property of the flow~\eqref{eq:Hpex},\eqref{e:hyperbolicity}.
\end{proof}

We now proceed as in~\cite[\S 7]{sj-z} and obtain
regularizations, $ \widehat \varphi_\pm $, of $ d ( \bullet,  
\Gamma_\pm )^2 $ -- see~\cite[Proposition~7.4]{sj-z}. The next lemma states
the properties of the resulting escape functions obtained using~\cite[Lemma~7.6]{sj-z}
applied with $ \epsilon = ( h/ \tilde h
)^{\frac12} $:
%
%
\begin{lemm}
  \label{l:f-1}
There exists a conic neighborhood $U'_0$ of $E_0^*$ and a real-valued function
$$
\hat f(x,\xi;h,\tilde h)\in C^\infty(U'_0\cap \{1/2<|\xi|<2\})
$$
such that:
\begin{enumerate}
\item $\hat f$ satisfies the derivative bounds
\begin{equation}
  \label{e:hat-f-bounded}
\hat f=\mathcal O(\log(1/h)),\quad
\partial^\alpha_{x,\xi}H_p^k \hat f=\mathcal O((h/\tilde h)^{-|\alpha|/2}),\
|\alpha|+k\geq 1;
\end{equation}
\item there exists a constant $C_{\hat f}$ such that
\begin{equation}
  \label{e:hat-f-positive}
H_p \hat f(x,\xi)\leq -C_{\hat f}^{-1}<0\text{ for }d((x,\xi),K)\geq C_{\hat f}(h/\tilde h)^{1/2}.
\end{equation}
\end{enumerate}
\end{lemm}


\subsection{A global escape function}
\label{s:global-escape}
We recall that our goal is to construct a function $ f $ such that for
the escape function, $ G $, given in Lemma \ref{l:escape-f-sj}
$ H_p ( G + f ) $ is as negative as possible. 

For that we cut $\hat f$ off and modify it to get an escape function defined
on the whole $T^*X$. Let $U'_0$ be the conic neighborhood of $E^*_0$ from
Lemma~\ref{l:f-1} and shrink it so that
$$
U'_0\cap (E_u^*\cup E_s^*)=\emptyset,\quad
U'_0\cap p^{-1}(1)\subset \{1/2<|\xi|<2\}.
$$
The second statement is possible since $E_0^*(x)\cap p^{-1}(1)=\alpha(x)$
and $|\alpha|=1$.
Take any conic neighborhood $U_0$ of $E_0^*$ such that $U_0\Subset U'_0$
and a nonnegative function
$$
\chi_{\hat f}\in C_0^\infty(U'_0\cap \{1/2<|\xi|<2\}),\quad
\chi_{\hat f}=1\text{ near }U_0\cap p^{-1}(1).
$$
Let $m\in S^0(X)$ be the function constructed in Lemma~\ref{l:escape-f-sj}.
We choose a constant $M>0$ large enough so that the function
\begin{equation}
  \label{e:f}
f:=\chi_{\hat f}\hat f+M\log(1/h)m
\end{equation}
satisfies
$$
\begin{gathered}
H_pf(x,\xi)\leq -c<0\text{ for }(x,\xi)\in U'_0\cap p^{-1}(1)
\text{ with }d((x,\xi),K)\geq C_{\hat f} (h/\tilde h)^{1/2},\\
H_p f\leq 0\text{ near }p^{-1}(1).
\end{gathered}
$$
This is possible since $H_p(\chi_{\hat f}\hat f)\leq -C_{\hat f}^{-1}<0$
when $(x,\xi)\in U_0\cap p^{-1}(1)$ and $d((x,\xi),K)\geq C_{\hat f}(h/\tilde h)^{1/2}$;
$H_p(\chi_{\hat f}\hat f)=\mathcal O(\log(1/h))$ everywhere by~\eqref{e:hat-f-bounded};
$\supp\chi_{\hat f}\subset U'_0$; and $H_pm\leq -c<0$ on $(U'_0\setminus U_0)\cap p^{-1}(1)$
by~\eqref{e:hpm}.

\begin{lemm}
\label{l:at} 
There exists a nonnegative function $\tilde a$ supported $\mathcal O((h/\tilde h)^{1/2})$
close to $K$ and such that for  $G$ given in Lemma~\ref{l:escape-f-sj}
and $ f $ given by \eqref{e:f}, 
\begin{equation}
  \label{e:positive}
H_p(G+f)-\tilde a\leq -c<0\text{ on }p^{-1}(1), \ \ \ \partial^\alpha
\tilde a = {\mathcal O} ( (h/\tilde h)^{-|\alpha|/2} ) .
\end{equation}
\end{lemm}

Equation~\eqref{e:positive} is the key component of the positive commutator
argument in \S\ref{s:mainproof}. By~\eqref{e:hpG} and the properties of $f$,
it suffices to verify~\eqref{e:positive} in an $\mathcal O((h/\tilde h)^{1/2})$ sized
neighborhood of~$K_1$, where $H_p(G+f)=H_p\hat f$ (since $m=0$ near $K_1$).

\medskip
\noindent
{\em Proof of Lemma \ref{l:at}.}
To construct~$\tilde a$, take a nonnegative function $\theta\in C^\infty(\mathbb R)$ such that
$\supp\theta\subset (-\infty,C_{\hat f}^{-1})$ and $\theta(\lambda)+\lambda=1$ for $\lambda\leq C_{\hat f}^{-1}/2$.
Then by~\eqref{e:hat-f-positive}, the function $\theta(-H_p\hat f)$ is supported
$\mathcal O((h/\tilde h)^{1/2})$ close to $K$. Now, take any nonnegative $\chi_a\in C_0^\infty(T^*X)$
such that $\chi_{\hat f}=1$ near $\supp\chi_a$, but $\chi_a=1$ near $U_0\cap p^{-1}(1)$, and define
\begin{equation}
  \label{e:tilde-a}
\tilde a:=\theta(-H_p\hat f)\chi_{a}\in C_0^\infty(T^*X).
\end{equation}
Then~\eqref{e:positive} follows since on $U_0\cap p^{-1}(1)$,
$$
H_p\hat f-\tilde a=H_p\hat f-\theta(-H_p\hat f)\leq -C_{\hat f}^{-1}/2.
$$
\stopthm

\section{Upper bound on the number of resonances}
\label{s:proof}

In this section, we prove the bound~\eqref{eq:b3} on the number of
Pollicott--Ruelle resonances.

\subsection{Reduction to a weighted estimate}
We start by showing how \eqref{eq:b3} follows from the estimate
\eqref{e:main}
given in the following lemma. 
%
%
\begin{lemm} (Main lemma)
  \label{l:main}
There exist families of bounded operators%
\footnote{To combine the notation of~\cite{f-sj} and~\cite{da-dy},
we denote by $G$ and $\hat f$ their respective escape functions and
by $G^w$ and $\widehat F$ the corresponding pseudodifferential operators.}
$\widehat F,F_1,A$ on $L^2(X)$, depending on two parameters $h,\tilde h$
(where we choose $\tilde h$ small enough and $h$ small enough depending on $\tilde h$),
such that for any fixed constant
$C_0$ and $t>0$ large enough, the modified conjugated operator
$$
\widetilde P_t:=e^{t\widehat F}e^{tF_1}P_{tG}e^{-tF_1}e^{-t\widehat F}-ith A
$$
satisfies the estimate (with $C$ independent of $h,\tilde h$) for any $u\in C^\infty(X)$
\begin{gather}
  \label{e:main}
\begin{gathered}
\|u\|_{L^2}\leq {C\over\max(h,\Imag z)}\|(\widetilde P_t-z)u\|_{L^2},
\\
\text{ for } \ 
|\Re z-1|\leq C_0h,\ \ 
-C_0h\leq\Im z\leq 1.
\end{gathered}
\end{gather}
Moreover, we can write $A=A_R+A_E$, where for some constant $C(\tilde h)$
depending on $\tilde h$,
\begin{gather}
  \label{e:finiterank}
\begin{gathered}
\|A_R\|_{L^2\to L^2}=\mathcal O(1),\quad
\|A_E\|_{L^2\to L^2}=\mathcal O(\tilde h),\\
\rank A_R\leq C(\tilde h)h^{-{n-1\over 2}}.
\end{gathered}
\end{gather}
\end{lemm}
%
%
Note that in~\cite{da-dy} we required the estimate~\eqref{e:main} for the
$H_h^{-1/2}\to H_h^{1/2}$ norm instead of the $L^2\to L^2$ norm; this is
because the Laplacian considered there is a differential operator
of order 2, while our differential operator $P$ has order 1.

Assume that \eqref{e:main} holds.
Since $e^{\pm t\widehat F},e^{\pm tF_1}$ are bounded on $L^2$, the operator
\begin{equation}
  \label{e:p-t}
P_t:=e^{t\widehat F}e^{tF_1}P_{tG}e^{-tF_1}e^{-t\widehat F}
\end{equation}
satisfies part (i) of Proposition~\ref{p:1}, and its eigenvalues in
$D(1,C_0h)$ are precisely the Pollicott--Ruelle resonances.
The operator
$A$ will be compactly microlocalized in the sense of~\cite[\S 3.1]{da-dy}
and in particular compact $L^2\to L^2$; therefore,
adding it will not change the Fredholm property of $P_t$.
By~\cite[Lemma~A.1]{f-sj}, the bound~\eqref{e:main}
implies
$$
\|(\widetilde P_t-z)^{-1}\|_{L^2\to L^2}\leq {C\over\max(h,\Imag z)},\quad
|\Re z-1|\leq C_0h,\
-C_0h\leq\Im z\leq 1.
$$
The estimate~\eqref{eq:b3} is now proved as in~\cite[\S 2]{da-dy}, using Jensen's inequality.

\smallskip

We now construct the operators $\widehat F,F_1,A$ of Lemma~\ref{l:main}.
We will use the class $\Psi_{1/2}^{\comp}(X)$ of pseudodifferential operators whose
symbols have compact essential support and satisfy the bound
$$
\sup|\partial^\alpha_{x,\xi}a|=\mathcal O((h/\tilde h)^{-|\alpha|/2}).
$$
We refer to \cite[\S 3.3]{sj-z} and \cite[\S 5.1]{da-dy} for the motivation for this class of symbols
and the properties of corresponding operators. We take 
$$
\widehat F:=(\chi_{\hat f}\hat f)^w,\quad
F_1:=M\log(1/h)m^w,
$$
so that the operator $\widehat F+F_1$ has the symbol $f$ from~\eqref{e:f}.
Recalling the derivative bounds~\eqref{e:hat-f-bounded}, we see that
$$
\widehat F\in\log(1/h)\Psi_{1/2}^{\comp}(X),\quad
F_1\in\log(1/h)\Psi^0(X).
$$
Finally, we put
\begin{equation}
  \label{e:A}
A:=\chi((\tilde h/h)\widehat P)\widetilde A,
\end{equation}
where:
\begin{itemize}
\item $\widetilde A=\tilde a^w$, with $\tilde a$ defined in~\eqref{e:tilde-a}.
By~\eqref{e:hat-f-bounded} and~\eqref{e:tilde-a}, we have
$\widetilde A\in\Psi^{1/2}(X)$;
\item $\chi\in C_0^\infty(\mathbb R)$ is equal to 1 near zero;
\item $\widehat P$ is any symmetric pseudodifferential operator in $\Psi^1(X)$
with principal symbol $\hat p(x,\xi)$ elliptic in the class $S^1$
for $|\xi|$ large enough and $\hat p=p-1$ on $U'_0\cap \{1/2<|\xi|<2\}\supset\supp\tilde a$;
\item $\chi((\tilde h/h)\widehat P)$ is defined by means of functional calculus
of self-adjoint operators on $L^2(X)$ (see~\cite[\S 5.2]{da-dy} for properties
of such operators).
\end{itemize}
Under these conditions, \eqref{e:finiterank} follows from~\cite[Lemma~6.1]{da-dy}. The key observation here is that $\tilde a$ is supported in an
$\mathcal O((h/\tilde h)^{1/2})$ sized neighborhood of $K$; the latter is an $n+1$ dimensional
smooth manifold invariant under the flow $\exp(tH_p)$ and thus under $\exp(tH_{\hat p})$
near $\supp\tilde a$; therefore, for each $R>0$ (see~\cite[\S 7.4]{da-dy})
$$
\begin{gathered}
\Vol_{\hat p^{-1}(0)}\{\exp(tH_{\hat p})(x,\xi)\mid |t|\leq R,\\
(x,\xi)\in(\supp\tilde a\cap \hat p^{-1}(0))+B_{\hat p^{-1}(0)}(R(h/\tilde h)^{1/2})\}
\leq C(h/\tilde h)^{n-1\over 2}.
\end{gathered}
$$

\subsection{Proof of Main Lemma \ref{l:main}}
  \label{s:mainproof}

In this section, we assume that $|\Re z -1|\leq C_0h$ and $-C_0h\leq\Im z\leq 1$,
and $u\in C^\infty(X)$; we will prove the estimate~\eqref{e:main}.
See the outline of the proof
of Theorem~2 in~\cite[Introduction]{da-dy} for an explanation of the positive commutator
argument used here.

We start by writing an expansion for the operator $P_t$ from~\eqref{e:p-t}. By~\eqref{eq:PtG},
$$
P_{tG}=P+t[G^w,P]+\mathcal O_t(h^2)_{\Psi^{-1+}}.
$$
Similarly,
$$
e^{tF_1}P_{tG}e^{-tF_1}=P+t[G^w+F_1,P]+\mathcal O_t(h^{2-})_{\Psi^{-1+}}.
$$
Finally, using the Bony--Chemin Theorem 
\cite[Th\'eor\`eme~6.4]{b-c},\cite[Theorem~8.6]{e-z} as in~\cite[Lemma~7.2]{da-dy}, we have
\begin{equation}
  \label{e:conj}
P_t=P+t[G^w+F_1+\widehat F,P]+\mathcal O_t(h\tilde h)_{L^2\to L^2}.
\end{equation}
In particular,
\begin{equation}
  \label{e:conj-rough}
P_t=P+\mathcal O_t(h)_{\Psi^{0+}}+\mathcal O_t(h\log(1/h))_{L^2\to L^2}.
\end{equation}

\smallskip

Next, we get rid of the $\chi((\tilde h/h)\widehat P)$ part
of the operator $A$; namely, we claim that~\eqref{e:main} follows from the estimate
\begin{equation}
  \label{e:main-1}
\|u\|_{L^2}\leq {C\over\max(h,\Im z)}\|(P_t-ith\widetilde A-z)u\|_{L^2}.
\end{equation}
For that, write $1-\chi(\lambda)=\lambda\psi(\lambda)$ with $\psi\in C^\infty(\mathbb R)$ bounded;
then
$$
\|(A-\widetilde A)u\|_{L^2}\leq C(\tilde h/h)\|\widehat P\widetilde Au\|_{L^2}.
$$
By~\eqref{e:hat-f-bounded} and~\eqref{e:tilde-a},
$H_p\tilde a=\mathcal O(1)_{S_{1/2}^{\comp}}$; therefore, by part~7 of~\cite[Lemma~5.2]{da-dy}
we have $[\widehat P,\widetilde A]=\mathcal O(h)_{L^2\to L^2}$, and
$$
\|(A-\widetilde A)u\|_{L^2}\leq C(\tilde h/h)\|\widetilde A\widehat Pu\|_{L^2}+\mathcal O(\tilde h)\|u\|_{L^2}.
$$
Since $\widehat P=P-1+\mathcal O(h)$ near $\WFh(\widetilde A)$, and by our assumptions on $z$
we find
$$
\|(A-\widetilde A)u\|_{L^2}\leq C(\tilde h/h)\|\widetilde A(P-z)u\|_{L^2}+\mathcal O(\tilde h\max(1,h^{-1}\Im z))\|u\|_{L^2}.
$$
Next, $\WFh(F_1)\cap\WFh(\widetilde A)=\emptyset$ since $\WFh(\widetilde A)\subset K$ and
$m=0$ in a conic neighborhood of $K$ by
Lemma~\ref{l:escape-f-sj}. Also, by~\eqref{e:hat-f-bounded}, 
\[ H_p(\chi_{\hat f}\hat f)=H_p\hat f=\mathcal O(1)_{S_{1/2}} \ \text{ 
near $\WFh(\widetilde A)\subset \{\chi_{\hat f}=1\} $.} \]
From part~7 of~\cite[Lemma~5.2]{da-dy}, we have $[P,\widehat F]=\mathcal O(h)$ near $\WFh(\widetilde A)$.
Then by~\eqref{e:conj}, we have $\widetilde A(\widetilde P_t-P)=\mathcal O_t(h)_{L^2\to L^2}$
and thus
$$
\|(A-\widetilde A)u\|_{L^2}\leq C(\tilde h/h)\|(\widetilde P_t-z) u\|_{L^2}+\mathcal O(\tilde h\max(1,h^{-1}\Im z))\|u\|_{L^2}.
$$
Combining this with~\eqref{e:main-1}, we get
$$
\|u\|_{L^2}\leq {C\over\max(h,\Im z)}\|(\widetilde P_t-z)u\|_{L^2}+\mathcal O(\tilde h)\|u\|_{L^2},
$$
which implies~\eqref{e:main} if $\tilde h$ is small enough.

\smallskip

To prove~\eqref{e:main-1}, we restrict to a neighborhood of the energy
surface as follows. Recalling~\eqref{e:positive}, the fact that
the function on the left-hand side of  \eqref{e:positive} is $\mathcal O(\log(1/h))_{S^{0+}}$,
and the only unbounded part of this function as $|\xi|\to\infty$ is $H_pG\leq 0$,
we see that there exists $B_E\in\Psi^0$ such that
$p-1$ is elliptic on $\WFh(B_E)$ and
\begin{equation}
  \label{e:positive2}
H_p(G+f)-\tilde a-\log(1/h)|\sigma(B_E)|^2\leq -c<0\text{ everywhere on }T^*X.  
\end{equation}
By the elliptic estimate (see for instance~\cite[(3.3)]{da-dy})
$$
\|B_E u\|_{L^2}\leq C\|(P-z)u\|_{H^{-1}_h}+\mathcal O(h^\infty)\|u\|_{L^2}.
$$
By~\eqref{e:conj-rough}, we have
\begin{equation}
\label{e:elliptic}
\|B_E u\|_{L^2}\leq C\|(P_t-ith\widetilde A-z)u\|_{L^2}+\mathcal O(h\log(1/h))\|u\|_{L^2}.
\end{equation}
We then claim that~\eqref{e:main-1} follows from
\begin{equation}
\label{e:main-2}
\Im\langle (P_t-ith(\widetilde A+\log(1/h)B_E^*B_E)) u,u\rangle_{L^2}\leq (-c_1th+\mathcal O_t(h\tilde h))\|u\|_{L^2}^2,
\end{equation}
where the constant $c_1>0$ is independent of $t$. Indeed, if $t$ is large enough
depending on $C_0$ and $\tilde h$ is small enough depending on $t$, then~\eqref{e:main-2} implies
$$
\Im\langle(P_t-ith\widetilde A-z)u,u\rangle_{L^2}\leq -C_t^{-1}\max(h,\Imag z)\|u\|_{L^2}^2
+th\log(1/h)\|B_Eu\|_{L^2}^2.
$$
Combining this with~\eqref{e:elliptic}, we get~\eqref{e:main-1}, which
is the claim in Lemma \ref{l:main}. 

\subsection{Proof of \eqref{e:main-2}}

Here we depart slightly from the strategy of \cite{da-dy} and replace
a microlocal partition of unity argument of \cite[\S 7]{da-dy} by global
positive commutator estimates.

 By~\eqref{e:conj} we reduce \eqref{e:main-2} to
$$
\Im\langle(P-t([P,G^w+F_1+\widehat F]+ih\widetilde A+ih\log(1/h)B_E^*B_E))u,u\rangle_{L^2}\leq -c_1th\|u\|_{L^2}^2.
$$
Since $P$ is self-adjoint on $L^2(X)$, this would follow from
\begin{equation}
\label{e:main-3}
\Re\langle Qu,u\rangle_{L^2}\geq -C\tilde h\|u\|_{L^2}^2,
\end{equation}
where
$$
Q:=-ih^{-1}[P,G^w+F_1+\widehat F]+\widetilde A-2c_1+\log(1/h)B_E^*B_E, 
$$
and 
\[ Q \in \Psi^{0+}+\log(1/h)\Psi^0+\log(1/h)\Psi^{\comp}_{1/2}. \]
Its principal symbol is given by 
$$
q:=-H_p(G+f)+\tilde a-2c_1+\log(1/h)|\sigma(B_E)|^2.
$$
Note that $Q$ is equal to any quantization of $q$ plus a remainder
that is $\mathcal O(\tilde h)_{L^2\to L^2}$. (See part~3 of~\cite[Lemma~5.4]{da-dy} for
the term involving $\widehat F$.)
Therefore, we can replace $Q$ by any quantization of $q$ in~\eqref{e:main-3}. 
Using~\eqref{e:positive2}, choose $c_1$
small enough so that
$$
q\geq c_1>0\text{ everywhere}.
$$
Formally speaking, \eqref{e:main-3} is a version of the sharp G\r arding inequality,
however the symbol involved is exotic and grows like $\mathcal O(\log(1/h))$, therefore we have to
break it into pieces using a partition of unity.
Note that, with a correct choice of $B_E$, by~\eqref{e:hpm} we have
$$
\begin{gathered}
\supp(\chi_{\hat f})\cap\supp(1-\chi_{\hat f})
\subset (U'_0\cap\{1/2<|\xi|<2\})\setminus (U_0\cap p^{-1}(1))\\
\subset \{H_pm>0\}\cup (T^*X\setminus p^{-1}(1))
\subset\{H_pm>0\}\cup \{\sigma(B_E)\neq 0\}.
\end{gathered}
$$
Therefore, we can write $T^*X=\Omega_0\cup\Omega_1\cup\Omega_2$, where $\Omega_j$ are open and
$$
\begin{gathered}
\chi_{\hat f}=1\text{ near }\Omega_0,\\
-MH_pm+|\sigma(B_E)|^2\geq c>0\text{ on }\Omega_1,\\
\chi_{\hat f}=0\text{ near }\Omega_2.
\end{gathered}
$$
We also make $\Omega_0$ and $\Omega_1$ bounded sets.
Now, take a partition of unity
$$
1=\chi_0+\chi_1+\chi_2,\quad
\chi_j\in C^\infty(T^*X;[0,1]),\quad
\supp\chi_j\subset\Omega_j.
$$
We use this partition to decompose $q$ into a sum of three symbols, except that the term
$-\chi_{\hat f}H_p\hat f+\tilde a$ will be put entirely into the part corresponding to $\Omega_0$.
Namely, put
$$
\begin{gathered}
q_0:=\chi_0 q+(1-\chi_0)(-\chi_{\hat f}H_p\hat f+\tilde a)+\chi_1|\sigma(B_E)|^2,\\
q_1:=\chi_1(q+\chi_{\hat f}H_p\hat f-\tilde a-|\sigma(B_E)|^2),\\
q_2:=\chi_2q.
\end{gathered}
$$
Since $\chi_2(\chi_{\hat f}H_p\hat f-\tilde a)=0$, we have
$$
q=q_0+q_1+q_2.
$$
Since $q_2\in \log(1/h)S^{0+}$, the sharp G\r arding inequality \cite[Theorem~9.11]{e-z} implies
\begin{equation}
  \label{e:positive-q2}
\langle q_2^w u,u\rangle\geq -Ch\log(1/h)\|u\|_{L^2}^2.
\end{equation}
Next, we consider the term corresponding to $q_1\in\log(1/h)S^{\comp}_{1/2}$, which we write
as
$$
q_1=\chi_1(\log(1/h)(-MH_pm+|\sigma(B_E)|^2)-\hat fH_p\chi_{\hat f}-H_pG-2c_1-|\sigma(B_E)|^2).
$$
Since $-MH_pm+|\sigma(B_E)|^2>0$ on $\Omega_1$,
we can increase $M$ and $|\sigma(B_E)|$ to make $q_1\geq c\log(1/h)\chi_1$. We will show that
\begin{equation}
  \label{e:positive-q1}
\langle q_1^wu,u\rangle\geq -Ch\log(1/h)\|u\|_{L^2}^2.
\end{equation}
Note that $q_1 + \chi_1 \hat fH_p\chi_{\hat f} \in \log(1/h) S^0$. To exploit this, put
$$
\Omega=T^*X\setminus(\Omega_0\cup\Omega_2),
$$
so that $\Omega$ is a compact set contained in $\Omega_1$. Since $\chi_1=1$ on $\Omega$,
we find $q_1\geq c\log(1/h)>0$ there.
Therefore, there exists $\chi_\Omega\in C_0^\infty(T^*X)$ such that $\chi_\Omega\neq 0$
near $\Omega$, but $q_1\geq \log(1/h)|\chi_\Omega|^2$ everywhere. We now apply the 
sharp G\r arding inequality for the $\Psi_{1/2}$ calculus, which follows from
the usual sharp G\r arding inequality by the standard rescaling (see for
example the proof of~\cite[Lemma~3.5]{sj-z}),
to the symbol 
$$ q_1-\log(1/h)|\chi_\Omega|^2\in\log(1/h)S_{1/2}^{\comp} . $$ 
Since the only exotic term in $q_1$
is $-\chi_1\hat f H_p\chi_{\hat f}$, supported in $\Omega$, we have
$$
\langle q_1^wu,u\rangle-\log(1/h)\|\chi_\Omega^wu\|_{L^2}^2\geq 
-Ch\log(1/h)\|u\|_{L^2}^2-C\tilde h\log(1/h)\|\chi_\Omega^w u\|_{L^2}^2.
$$
For $\tilde h$ small enough, this yields \eqref{e:positive-q1}.
Now, we write $q_0$ as a sum of two terms, one non-exotic and one compactly microlocalized:
$$
\begin{gathered}
q_0=q'_0+q''_0,\\
q'_0=\chi_0(-H_pG-M\log(1/h)H_pm+(\log(1/h)-1)|\sigma(B_E)|^2),\\
q''_0=-\chi_{\hat f}H_p\hat f+\tilde a-2\chi_0 c_1+(1-\chi_2)|\sigma(B_E)|^2.
\end{gathered}
$$
Then $q'_0\in \log(1/h)S^{0}$ and $q'_0\geq 0$ everywhere
(increasing $|\sigma(B_E)|$ if necessary to handle the set $\{|\xi|\leq 1/2\}$);
by sharp G\r arding inequality~\cite[Theorem~4.32]{e-z} applied
to the symbol $q'_0/\log(1/h)$,
\begin{equation}
  \label{e:positive-q0'}
\langle (q'_0)^wu,u\rangle\geq -Ch\log(1/h)\|u\|_{L^2}^2.
\end{equation}
Next, $q''_0\in S^{\comp}_{1/2}$ and, if we choose the function $\chi_a$
from the definition~\eqref{e:tilde-a} of $\tilde a$ so that $\chi_a=1$ near
$\supp\chi_0$, and take $c_1$ small enough, we have $q''_0\geq 0$ everywhere.
Again using the sharp G\r arding inequality for $\Psi_{1/2}$ calculus,
we find
\begin{equation}
  \label{e:positive-q0''}
\langle (q''_0)^wu,u\rangle\geq -C\tilde h\|u\|_{L^2}^2.
\end{equation}
Adding together~\eqref{e:positive-q2}, \eqref{e:positive-q1},
\eqref{e:positive-q0'}, and~\eqref{e:positive-q0''}, we get~\eqref{e:main-3}.

\smallsection{Acknowledgements} 
We are grateful to Fr\'ed\'eric Faure for helpful comments, in
particular on the optimality of  polynomial bounds, and
to an anonymous referee for useful remarks.
We also would like to acknowledge partial
support by the National Science Foundation from a postdoctoral
fellowship (KD) and 
the grant DMS-1201417 (SD, MZ). 
%


\def\arXiv#1{\href{http://arxiv.org/abs/#1}{arXiv:#1}}


\begin{thebibliography}{0}

\bibitem[Ba]{bal} Viviane Baladi,
	\emph{Anisotropic Sobolev spaces and dynamical transfer operators: $C^\infty$ foliations,\/}
	Algebraic and topological dynamics, 123--135,
	Contemp. Math. \textbf{385}, AMS, 2005.
	
\bibitem[BaTs]{b-t} Viviane Baladi and Masato Tsujii,
	\emph{Anisotropic H\"older and Sobolev spaces for hyperbolic diffeomorphisms,\/}
	Ann. Inst. Fourier \textbf{57}(2007), no.~1, 127--154.

\bibitem[B\'e]{Ber} Pierre B\'erard,
	\emph{On the wave equation on a compact Riemannian manifold without conjugate points,\/}
	Math. Z. \textbf{155}(1977), no.~3, 249--276.

\bibitem[BlKeLi]{b-k-l} Michael Blank, Gerhard Keller, and Carlangelo Liverani,
	\emph{Ruelle--Perron--Frobenius spectrum for Anosov maps,\/}
	Nonlinearity \textbf{15}(2002), no.~6, 1905--1973.

\bibitem[BoCh]{b-c} Jean-Michel Bony and Jean-Yves Chemin,
	\emph{Espaces fonctionnels associ\'es au calcul de Weyl--H\"ormander,\/}
	Bull. Soc. Math. France \textbf{122}(1994), no.~1, 77--118.

\bibitem[BuLi]{b-l} Oliver Butterley and Carlangelo Liverani,
	\emph{Smooth Anosov flows: correlation spectra and stability,\/}
	J. Mod. Dyn. \textbf{1}(2007), no.~2, 301--322.

\bibitem[DaDy]{da-dy} Kiril Datchev and Semyon Dyatlov,
	\emph{Fractal Weyl laws for asymptotically hyperbolic manifolds,\/}
	preprint, \arXiv{1206.2255v2}.

	



\bibitem[FaRoSj]{f-r-sj} Fr\'ed\'eric Faure, Nicolas Roy, and Johannes Sj\"ostrand,
	\emph{A semiclassical approach for Anosov diffeomorphisms and Ruelle resonances,\/}
	Open Math. Journal \textbf{1}(2008), 35--81.

\bibitem[FaSj]{f-sj} Fr\'ed\'eric Faure and Johannes Sj\"ostrand,
	\emph{Upper bound on the density of Ruelle resonances for Anosov flows,\/}
	Comm. Math. Phys. \textbf{308}(2011), no.~2, 325--364.

\bibitem[FaTs]{fat} Fr\'ed\'eric Faure and Masato Tsujii,
	\emph{Prequantum transfer operator for Anosov diffeomorphism (Preliminary Version),\/}
	preprint, \arXiv{1206.0282}.

\bibitem[GoLi]{g-l} S\'ebastien Gou\"ezel and Carlangelo Liverani,
	\emph{Banach spaces adapted to Anosov systems,\/}
	Ergodic Theory Dynam. Systems \textbf{26}(2006), no.~1, 189--217.

\bibitem[HeSj]{h-sj} Bernard Helffer and Johannes Sj\"ostrand, 
	\emph{R\'esonances en limite semi-classique (Resonances in semi-classical limit),\/}
	M\'emoires de la S.M.F., \textbf{24/25}, 1986.

\bibitem[H\"o3]{hor} Lars H\"ormander,
	\emph{The Analysis of Linear Partial Differential Operators, Volume III,\/}
	Springer, 1985.

\bibitem[Le]{Leb} Patricio Leboeuf,
	\emph{Periodic orbit spectrum in terms of Ruelle--Pollicott resonances,\/}
	Phys. Rev. E (3) {\bf 69}, no.~2, 026204 (2004).

\bibitem[Li04]{liv1} Carlangelo Liverani,
	\emph{On contact Anosov flows,\/}
	Ann. of Math. (2) \textbf{159}(2004), no.~3, 1275--1312.
	
\bibitem[Li05]{liv2} Carlangelo Liverani,
	\emph{Fredholm determinants, Anosov maps and Ruelle resonances,\/}
	Discrete Contin. Dyn. Syst. \textbf{13}(2005), no.~5, 1203--1215.

\bibitem[Me]{Mel3} Richard B. Melrose,
	\emph{Polynomial bounds on the distribution of poles in scattering by an obstacle,\/}
	Journ\'ees ``\'Equations aux D\'eriv\'ees partielles'', Saint-Jean de Monts, 1984.\\
	\url{http://archive.numdam.org/article/JEDP_1984____A3_0.djvu}

\bibitem[NoSjZw]{sj-n-z} St\'ephane Nonnenmacher, Johannes Sj\"ostrand, and Maciej Zworski,
	\emph{Fractal Weyl law for open quantum chaotic maps,\/}
	preprint, \arXiv{1105.3128}.

\bibitem[Po]{Po} Mark Pollicott, {\em On the rate of mixing of Axiom A flows,}
Inv. Math. {\bf 81}(1986), 147--164.

\bibitem[Ra]{Ran} Burton Randol,
	\emph{The Riempann hypothesis for Selberg's zeta-function and the asymptotic behavior of eigenvalues of the Laplace operator,\/}
 	Trans. Amer. Math. Soc. {\bf 236}(1978), 209--223. 

\bibitem[Ru]{Ru} David Ruelle, {\em Resonances of chaotic dynamical
    systems}, Phys.~Rev.~Lett. {\bf 56}, 405-407.

\bibitem[Sj]{sj-d} Johannes Sj\"ostrand,
	\emph{Geometric bounds on the density of resonances for semiclassical problems,\/}
	Duke Math. J. \textbf{60}(1990), no.~1, 1--57.

\bibitem[SjZw]{sj-z} Johannes Sj\"ostrand and Maciej Zworski,
	\emph{Fractal upper bounds on the density of semiclassical resonances,\/}
	Duke Math. J. \textbf{137}(2007), no.~3, 381--459.

\bibitem[Ts08]{ts1} Masato Tsujii,
	\emph{Decay of correlations in suspension semi-flows of angle-multiplying maps,\/}
	Ergodic Theory Dynam. Systems \textbf{28}(2008), no.~1, 291--317.

\bibitem[Ts10a]{ts2} Masato Tsujii,
	\emph{Quasi-compactness of transfer operators for Anosov flows,\/}
	Nonlinearity \textbf{23}(2010), no.~7, 1495--1545.

\bibitem[Ts10b]{ts3} Masato Tsujii,
	\emph{Contact Anosov flows and the FBI transform,\/}
	preprint, \arXiv{1010.0396}.

\bibitem[Vo]{Vo} Georgi Vodev,
	\emph{Sharp bounds on the number of scattering poles in even-dimensional spaces,\/}
	Duke Math. J. {\bf 74} (1994), 1--17.

\bibitem[Zw1]{Zw} Maciej Zworski, 
	\emph{Sharp polynomial bounds on the number of scattering poles,\/}
	Duke Math. J. \textbf{59}(2)(1989), 311--323.

\bibitem[Zw]{e-z} Maciej Zworski,
	\emph{Semiclassical analysis,\/}
	Graduate Studies in Mathematics \textbf{138} AMS, 2012.

\end{thebibliography}
\end{document}